\documentclass[12pt]{article}
\usepackage{amsmath,amssymb,amsthm}

\newtheorem{theorem}{Theorem}

\newtheorem{corollary}[theorem]{Corollary}

\def\barr{\begin{array}}
\def\earr{\end{array}}

\title{Breaking points in centralizer lattices}
\author{Marius T\u arn\u auceanu}
\date{January 4, 2018}

\begin{document}

\maketitle

\begin{abstract}
    In this note, we prove that the centralizer lattice ${\mathfrak C}(G)$ of a group $G$ cannot be written 
    as a union of two proper intervals. In particular, it follows that ${\mathfrak C}(G)$ has no breaking 
    point. As an application, we show that the generalized quaternion $2$-groups are not capable.
\end{abstract}

{\small
\noindent
{\bf MSC 2010\,:} Primary 20D30; Secondary 20D15, 20E15.

\noindent
{\bf Key words\,:} breaking point, centralizer lattice, interval, generalized quaternion $2$-group, capable group.}

\section{Introduction}

Let $G$ be a finite group and $L(G)$ be the subgroup lattice of $G$. The starting point for our discussion is given by \cite{2},
where the proper nontrivial subgroups $H$ of $G$ satisfying the condition
$$\text{for every}\hspace{1mm} X\in L(G)\hspace{1mm} \text{we have either}\hspace{1mm} X\leq H\hspace{1mm} \text{or}\hspace{1mm} H\leq X \leqno(1)$$have been studied. Such a subgroup is called a \textit{breaking point} for the lattice $L(G)$, and a group $G$ whose subgroup lattice possesses breaking points is called a \textit{BP-group}. Clearly, all cyclic $p$-groups of order at least $p^2$ are BP-groups. Note that a complete classification of BP-groups can be found in \cite{2}. Also, we observe that the condition (1) is equivalent to $$L(G)=[1,H]\cup[H,G],\leqno(2)$$where for $X,Y\in L(G)$ with $X\subseteq Y$ we denote by $[X,Y]$ the interval in $L(G)$ between $X$ and $Y$. A natural generalization of (2) has been suggested by Roland Schmidt, namely
$$L(G)=[1,M]\cup[N,G] \mbox{ with } 1<M,N<G,\leqno(3)$$and the abelian groups $G$ satisfying (3) have been determined in \cite{1}.

The above concepts can be naturally extended to other remarkable posets of subgroups of $G$, and also to
arbitrary posets. We recall here that the generalized quaternion $2$-groups
$$Q_{2^n}=\langle a,b \mid a^{2^{n-2}}= b^2, a^{2^{n-1}}=1, b^{-1}ab=a^{-1}\rangle, n\geq 3$$can
be characterized as being the unique finite non-cyclic groups whose posets of cyclic subgroups
and of conjugacy classes of cyclic subgroups have breaking points (see \cite{7} and \cite{3}, respectively).

In the current note, we will focus on the centralizer lattice $${\mathfrak C}(G)=\{C_G(H) \,|\, H\in L(G)\}$$of $G$. Note that this is a complete 
meet-sublattice of $L(G)$ with the least element $Z(G)=C_G(G)$ and the greatest element $G=C_G(1)$. We will prove that there are no proper 
centralizers $M$ and $N$ such that ${\mathfrak C}(G)=[Z(G),M]\cup[N,G]$. This implies that ${\mathfrak C}(G)$ does not have breaking points. As an application, we show that $Q_{2^n}$ is not a capable group, i.e. there is no group $G$ with $G/Z(G)\cong Q_{2^n}$ (see e.g. the main theorem of \cite{6}).

Most of our notation is standard and will usually not be repeated here. Elementary concepts and results on group theory 
can be found in \cite{4}. For subgroup lattice notions we refer the reader to \cite{5} .

\section{Main results}

Our main theorem is the following.

\begin{theorem}
    Let $G$ be a group and ${\mathfrak C}(G)$ be the centralizer lattice of $G$. Then ${\mathfrak C}(G)$ cannot be written as ${\mathfrak C}(G)=[Z(G),M]\cup[N,G]$ with $M,N\neq Z(G),G$.
\end{theorem}

\begin{proof}
   Assume that there are two proper centralizers $M$ and $N$ such that ${\mathfrak C}(G)=[Z(G),M]\cup[N,G]$. Then for every $x\in G$ we have either $C_G(x)\leq M$ or $N\leq C_G(x)$. In the first case we infer that $x\in M$, while in the second one we get $x\in C_G(C_G(x))\leq C_G(N)$, that is $x\in C_G(N)$. Thus, the group $G$ is the union of its proper subgroups $M$ and $C_G(N)$, a contradiction.
\end{proof}

Clearly, by taking $M=N$ in Theorem 1, we obtain the following corollary.

\begin{corollary}
    The centralizer lattice ${\mathfrak C}(G)$ of a group $G$ has no breaking point.
\end{corollary}

Next we remark that for an abelian group $G$ we have ${\mathfrak C}(G)=\{G\}$, and also that there is no non-abelian group $G$ with ${\mathfrak C}(G)=\{Z(G),G\}$ (i.e. ${\mathfrak C}(G)$ is not a chain of length $1$). Since chains of length at least $2$ have breaking points, Corollary 2 implies that:

\begin{corollary}
    The centralizer lattice ${\mathfrak C}(G)$ of a group $G$ cannot be a chain of length $\geq 1$. Moreover, ${\mathfrak C}(G)$ is a chain if and only if $G$ is abelian.
\end{corollary}

Another consequence of Corollary 2 is:

\begin{corollary}
    The generalized quaternion $2$-groups $Q_{2^n}$, $n\geq 3$, are not capable groups.
\end{corollary}

\begin{proof}
   Assume that there is a group $G$ such that $G/Z(G)\cong Q_{2^n}$. Obviously, $G$ is not abelian. Since $Q_{2^n}$ has a unique subgroup of order $2$, it follows that the lattice interval $[Z(G),G]$ contains a unique minimal element, say $H$. If $H\in {\mathfrak C}(G)$, then it is a breaking point of ${\mathfrak C}(G)$, contradicting Corollary 2. If $H\notin {\mathfrak C}(G)$, then it is (properly) contained in all minimal centralizers $M_1$, $M_2$, ..., $M_k$ of $G$, and so $H\subseteq\bigcap_{i=1}^k M_i$. Note that a intersection of centralizers is also a centralizer, that is $\bigcap_{i=1}^k M_i\in {\mathfrak C}(G)$. On the other hand, we have $k\geq 3$ because $G$ is non-abelian. Then $\bigcap_{i=1}^k M_i< M_j$, for any $j=1,2, ..., k$, and therefore $\bigcap_{i=1}^k M_i=Z(G)$ by the minimality of $M_j$'s. Consequently, $H\subseteq Z(G)$, a contradiction.
\end{proof}

Finally, we formulate an open problem concerning the above study.

\bigskip\noindent {\bf Open problem.} Let $G$ be a group. Then ${\mathfrak C}'(G)=\{C_G(H) \,|\, H\unlhd G)\}$ is also a complete
meet-sublattice of $L(G)$ with the least element $Z(G)=C_G(G)$ and the greatest element $G=C_G(1)$. Which are the groups $G$ such that ${\mathfrak C}'(G)$ has breaking points? (note that this can happen, as for $G=S_3$)

\vspace*{5ex}\small

\hfill
\begin{minipage}[t]{5cm}
Marius T\u arn\u auceanu \\
Faculty of  Mathematics \\
``Al.I. Cuza'' University \\
Ia\c si, Romania \\
e-mail: {\tt tarnauc@uaic.ro}
\end{minipage}

\end{document}